\documentclass[reqno]{amsart}
\usepackage{enumitem}

\usepackage[utf8x]{inputenc}  % this enables accented characters

\usepackage{listings}
\usepackage{caption}
\usepackage{easytable}

\usepackage{setspace}
\usepackage[left=3.1cm, right=3.1cm, bottom=4cm]{geometry}                 

\usepackage{graphicx} % this enables figures
\usepackage{pgf,tikz}
\usetikzlibrary{arrows}
\usepackage{amssymb}
\usepackage{pdfsync}
\usepackage{mathrsfs}
\usepackage{hyperref} % this enables hyperlinks

\usepackage{epstopdf}
\usepackage{bbm} 
\usepackage[colorinlistoftodos,prependcaption,textsize=tiny]{todonotes}
\usepackage{xargs}
\usepackage{bm}
\usepackage{lmodern}

\usepackage{listings} % code highlighting

\def\R{\mathbb{R}}

\def\Z{\mathbb{Z}}
\def\C{\mathbb{C}}

\def\F{\mathbb{F}}

\def\A{\mathcal{A}}

\def\E{\mathcal{E}}

\def\P{\mathcal{P}}

\def\X{\mathcal{X}}

\renewcommand{\d}{\text{\rm d}}

\newcommand{\Fq}{\mathbb{F}_q}
\newcommand{\Pd}{\mathcal{P}_d}
\newcommand{\Prob}{\mathbb{P}}
\newcommand{\Simp}{\Delta_q}
\DeclareMathOperator{\Bin}{Bin}
\DeclareMathOperator{\Sym}{Sym}
\DeclareMathOperator{\Tr}{Tr}
\newcommand{\qtq}[1]{\quad\text{#1}\quad}

\DeclareRobustCommand{\rchi}{{\mathpalette\irchi\relax}}
\newcommand{\irchi}[2]{\raisebox{\depth}{$#1\chi$}} % inner command, used by \rchi

\newtheorem{theorem}{Theorem}
\newtheorem{conjecture}[theorem]{Conjecture}

\newtheorem{corollary}[theorem]{Corollary}
\newtheorem{definition}[theorem]{Definition}
\newtheorem*{definition*}{Definition}
\newtheorem{remark}[theorem]{Remark}
\newtheorem{proposition}[theorem]{Proposition}
\newtheorem{lemma}[theorem]{Lemma}

\newtheorem{example}[theorem]{Example}

\makeatletter
\DeclareFontFamily{U}{tipa}{}
\DeclareFontShape{U}{tipa}{m}{n}{<->tipa10}{}
\newcommand{\arc@char}{{\usefont{U}{tipa}{m}{n}\symbol{62}}}%

\makeatother

\numberwithin{equation}{section}

\allowdisplaybreaks

\newcommand{\intav}[1]{\mathchoice {\mathop{\vrule width 6pt height 3 pt depth  -2.5pt
\kern -8pt \intop}\nolimits_{\kern -6pt#1}} {\mathop{\vrule width
5pt height 3  pt depth -2.6pt \kern -6pt \intop}\nolimits_{#1}}
{\mathop{\vrule width 5pt height 3 pt depth -2.6pt \kern -6pt
\intop}\nolimits_{#1}} {\mathop{\vrule width 5pt height 3 pt depth
-2.6pt \kern -6pt \intop}\nolimits_{#1}}}

\newcommand{\intavl}[1]{\mathchoice {\mathop{\vrule width 6pt height 3 pt depth  -2.5pt
\kern -8pt \intop}\limits_{\kern -6pt#1}} {\mathop{\vrule width 5pt
height 3  pt depth -2.6pt \kern -6pt \intop}\nolimits_{#1}}
{\mathop{\vrule width 5pt height 3 pt depth -2.6pt \kern -6pt
\intop}\nolimits_{#1}} {\mathop{\vrule width 5pt height 3 pt depth
-2.6pt \kern -6pt \intop}\nolimits_{#1}}}

\title[Sharp extension for the moment curve on finite fields II]
{Sharp endpoint extension inequalities \\ for the moment curve on finite fields II: \\ an extremal property of the uniform distribution}

\author[Biswas]{Chandan Biswas}
\address{Department of Mathematics, Indian Institute of Technology Bombay, Mumbai 400076, India}
\email{cbiswas@iitb.ac.in}

\author[Carneiro]{Emanuel Carneiro}
\address{ICTP - The Abdus Salam International Centre for Theoretical Physics, 
Strada Costiera, 11, I - 34151, Trieste, Italy}
\email{carneiro@ictp.it}

\author[Flock]{Taryn C. Flock}
\address{Macalester College\\Mathematics, Statistics, and Computer Science\\
Olin-Rice Science Center, Room 222\\
Saint Paul, MN 55105-1899\\USA}
\email{tflock@macalester.edu}

\author[Madrid]{Jos\'{e} Madrid}
\address{Department of Mathematics, Virginia Polytechnic Institute and State University, 225 Stanger Street, Blacksburg, VA 24061-1026, USA}
\email{josemadrid@vt.edu}

\author[Oliveira e Silva]{Diogo Oliveira e Silva}
\address{ 
Center for Mathematical Analysis, Geometry and Dynamical Systems \&
Departamento de Matem\'{a}tica\\ 
Instituto Superior T\'{e}cnico\\
Universidade de Lisboa\\
Av.\@ Rovisco Pais\\ 
1049-001 Lisboa, Portugal} 
\email{diogo.oliveira.e.silva@tecnico.ulisboa.pt}

\author[Stovall]{Betsy Stovall}
\address{University of Wisconsin--Madison\\Department of Mathematics\\
480 Lincoln Drive\\ 
Madison, WI 53706\\
USA}
\email{stovall@math.wisc.edu}

\author[Tautges]{James Tautges}
\address{University of Pennsylvania\\Department of Mathematics\\
204 South 32nd Street \\
Philadelphia, PA 19104 \\
USA}
\email{tautges@sas.upenn.edu}

\begin{document}

\subjclass[2020]{42B10, 12E20, 05C70, 05B25, 26D15, 05E05}
\keywords{Sharp restriction theory, finite fields, moment curve, maximizers, uniform distribution}

\begin{abstract}
We identify the optimal constant and describe all maximizers for the Fourier endpoint extension inequality associated with the moment curve over finite fields. This confirms a conjecture proposed in the authors’ earlier work. The proof proceeds by showing that the problem is equivalent to a purely probabilistic extremal statement: among all probability distributions on a finite set, the uniform distribution uniquely maximizes the expected number of distinct rearrangements of an i.i.d.\@ sample.
\end{abstract}
\date{\today}

\maketitle

%\tableofcontents

\section{Introduction and statement of results}\label{Sec : intro}

\subsection{Setup}\label{Sec : setup}
We will follow the notation and framework used in ~\cite{Group2026, MockenhauptTao2004}. Let $d \ge 2$ be an integer, and let $\Fq$ be the field with $q$ elements, with $q = p^n$ for some positive integer $n$ with the prime $p > d$. The moment curve is
$$
\Gamma:=\{\gamma(t):t\in\Fq\}, \qtq{where} \gamma(t):=(t,t^2,\dots,t^d)\in\Fq^d.
$$
Let $\sigma$ be the normalized counting measure on $\Gamma$, i.e., the measure giving each of the $q$ points of $\Gamma$ mass $1/q$; thus, $\|g\|_{L^2(\Gamma,d\sigma)}^2 = \frac1q \sum_{\xi\in\Gamma}|g(\xi)|^2$ for a function $g : \Gamma \to \C$.

For $y \in \Fq$, let $\Tr : \Fq \to \F_p$ be the field trace,
$$
\Tr(y):=y+y^p+\dots+y^{p^{n-1}},
$$
viewed as an element of $\F_p$, and set $e(y) := \exp(2 \pi i \Tr(y)/p)$. For $x, \xi \in \Fq^d$, we write
$x \cdot \xi := \sum_{i = 1}^d x_i \xi_i \in \Fq$. We now state the object of our study. Given a function $f : \Gamma \to \mathbb C$, the Fourier extension map $(f \sigma)^\vee : \Fq^d \to \mathbb C$ is defined by
\begin{equation}\label{E : def extn}
(f\sigma)^\vee(x):=\frac1q\sum_{\xi\in\Gamma}f(\xi)\,e(x\cdot\xi).
\end{equation}

Let $\Pd$ denote the set of partitions of $d$ into at most $d$ nonnegative parts, written as non-increasing $d$-tuples $\ell=(\ell_1,\dots,\ell_d)$ with $\sum_i\ell_i=d$. For $1\le j\le d$, let
$$
b_j(\ell):=\#\{1\le i\le d:\ell_i=j\},\qquad b_0(\ell):=\#\{1\le i\le
d:\ell_i=0\}+(q-d).
$$
Write $b(\ell):=(b_0(\ell),\dots,b_d(\ell))$ and set
\begin{equation}\label{eq_A}
A:=\frac1{q^d}\sum_{\ell\in\Pd}\binom{d}{\ell}^2\binom{q}{b(\ell)}. 
\end{equation}

\begin{conjecture}[\cite{Group2026}]\label{Conj1}
Let $p >d$. With notations as above, the inequality 
\begin{equation}\label{E : conj1}
\|(f \sigma)^{\vee}\|_{L^{2d}(\F_q^d, \d x)} \leq A^{1/2d} \|f\|_{L^{2}(\Gamma, \d\sigma)}
\end{equation}
holds for every $f:\Gamma\to\mathbb C$ and is sharp. Moreover, a nonzero $f$ is a maximizer of \eqref{E : conj1} if and only if $|f|$ is constant.
\end{conjecture}

Here $\d x$ denotes the counting measure on $\Fq^d$, so that
$$
\|F\|_{L^{2d}(\Fq^d,\d x)}^{2d} := \sum_{x\in\Fq^d}|F(x)|^{2d}.
$$
The constant $A$ admits a transparent probabilistic description, recorded in Proposition~\ref{Prop : prob form}; in particular $A < d!$, with $A \to d!$ as $q\to\infty$ for each fixed $d$ (see Remark~\ref{rmk : size of A}).

In \cite{Group2026} we established the above conjecture in two regimes: for $2\le d\le 20$ and for $q\ge \tfrac{d(d-1)}{2\log 6}+\tfrac{2d-1}{3}$. The argument was to reduce the  bound~\eqref{E : conj1} to a finite combinatorial verification, which was then carried out explicitly for $d\le20$ and via a formal proof for $q \geq \frac{d(d-1)}{2 \log 6} + \frac{2d-1}{3}$. The purpose of this article is to prove the following theorem, which settles the conjecture in full generality.

\begin{theorem}\label{T : main thm}
Let $d\ge2$ and $q=p^n$ be a prime power with $p>d$. Then Conjecture \ref{Conj1} holds: inequality \eqref{E : conj1} is valid, the constant $A^{1/2d}$ is optimal, and a nonzero $f$ attains equality if and only if $|f|$ is constant.
\end{theorem}

Our route to Theorem~\ref{T : main thm} passes through a reformulation after which neither the moment curve nor the field $\Fq$ plays any role. Let $\mathcal X$ be a finite set, let $\theta$ be a probability distribution on $\mathcal X$, and let $T=(T_1,\dots,T_d)$ be i.i.d.\@ with law $\theta$. Write $M(T)$ for the number of distinct rearrangements of the sample $T$.

\begin{theorem}\label{T : prob thm}
Let $d\ge2$ and let $\mathcal X$ be a finite set. Then
$$
\E_\theta\big[M(T)\big]\le\E_{\theta_0}\big[M(T)\big]
$$
for every probability distribution $\theta$ on $\mathcal X$, where $\theta_0$ denotes the uniform distribution, with equality if and only if $\theta=\theta_0$.
\end{theorem}

For $\mathcal X=\Fq$, Theorem~\ref{T : prob thm} establishes validity of Conjecture~\ref{Conj1}; see Corollary~\ref{cor:equivalence}. Under this correspondence, the sharp constant is $A=\E_{\theta_0}[M(T)]$, and the equality case $\theta=\theta_0$ is exactly the condition that $|f|$ is constant. We prove Theorem~\ref{T : prob thm} by a two-point symmetrization: moving the masses at two letters of the alphabet towards one another strictly increases $\E_\theta[M(T)]$, so the uniform distribution is the only maximizer. Two-point arguments of this kind underlie Beckner's proof of the sharp Hausdorff-Young inequality~\cite{Beckner1975}, and here they replace the finite combinatorial verification of~\cite{Group2026}.

The study of optimal constants and maximizers for Fourier extension inequalities has a rich history in the euclidean setting. Foundational results are due to Foschi~\cite{Foschi2007}, who identified the maximizers of the Strichartz inequalities for the paraboloid (in dimensions $1$ and $2$) and the cone (in dimensions $2$ and $3$), and to  Christ--Shao~\cite{christ2012existence, christ2012extremizers}, who established the existence of maximizers for  the endpoint Stein--Tomas inequality on $\mathbb{S}^2$; see also Frank--Lieb--Sabin~\cite{FrankLiebSabin2016} for a sharp precompactness criterion for optimizing sequences. The subject has since flourished; see, for instance~\cite{CarneiroOeSSousa2019,FoschiSphere2015,GoncalvesNegro2022,Negro2023, NegroOeSStovallTautges2025}, and we refer to the survey~\cite{FoschiOeS2017} for a broader account.

Sharp questions of this type are particularly natural for the moment curve, whose euclidean restriction theory was settled by Drury~\cite{Drury1985} and which underlies the resolution of the Vinogradov mean value theorem~\cite{BDG2016}; see also~\cite{GLYZ2021}. There, however, the optimal constant is not known in any dimension: Biswas and Stovall established the existence of
maximizers~\cite{BiswasStovall2023} and obtained sharp restriction estimates for monomial curves~\cite{BiswasStovall2025}, but the maximizers themselves remain elusive. This is one of the motivations for studying the finite field model, where the arithmetic structure (see Section~\ref{subsec: newton identity}) makes exact computations possible.

The extension problem over finite fields was initiated by Mockenhaupt and Tao~\cite{MockenhauptTao2004}, in part because of its close relationship with the Kakeya and restriction problems~\cite{Dvir2009} (see also \cite[page 39]{MockenhauptTao2004}). Since then, substantial literature has developed both at the endpoint~\cite{LewkoLewko2012} and beyond it~\cite{IosevichKohLewko2020,KohLeePham2022, Lewko2015, Lewko2019, RudnevShkredov2018}. The two lines of research intersected only very recently, when Gonz\'alez-Riquelme and Oliveira e Silva~\cite{RuquelmeSilva2024} determined the optimal constants and the maximizers for the parabola, the paraboloid,  the hyperbolic paraboloid, and certain cones; further cones were subsequently settled by Gonz\'alez-Riquelme and Ismoilov~\cite{GonzalezIsmoilov2026}, and the moment curve was taken up in~\cite{Group2026}. With the sole exception of~\cite{Group2026}, all of the works just mentioned concern hypersurfaces. Our Theorem~\ref{T : main thm} completely resolves the optimal constant and the uniqueness of maximizers for the moment curve in every dimension.

\subsection{Outline of the proof}

The proof has three main parts.

\begin{itemize}[leftmargin=1.6em]
\item \textbf{Section~3.} For the convenience of the reader, and to keep the present account self-contained, we recall from \cite[Eq.~2.10]{Group2026} the reduction of Conjecture \ref{Conj1} to the equivalent statement that a certain explicit symmetric form $Q(x)$ in $q$ nonnegative real variables $x=(x_\xi)_{\xi\in\Fq}$ is nonnegative (with equality occurring only for constant vectors).

\item \textbf{Section~4.} We show that $Q(x)\ge0$ is exactly the statement that a certain expectation, taken over a probability distribution $\theta$ obtained
by normalizing $x$, is maximized when $\theta$ is uniform. Precisely, if $T=(T_1,\dots,T_d)$ are i.i.d.\@ draws from a distribution $\theta$ on a $q$-element set and $M(T)$ denotes the number of distinct orderings of the sample $T$, then $Q(x)\ge0$ holds for all $x$, with equality only at constant $x$, if and only if
$$
\E_\theta[M(T)]\le \E_{\theta_0}[M(T)]\qquad\text{for every probability
vector }\theta,\text{ with equality only at }\theta=\theta_0.
$$

\item \textbf{Section~5.} We prove this purely probabilistic statement. The proof merges two letters $a\ne b$ of the alphabet into one, which factors $M(T)$ exactly as a product of a binomial coefficient $\binom{S}{N_a}$ (where $S=N_a+N_b$) and the rearrangement count $M(\widetilde T)$ of the merged sample; taking expectations isolates the dependence on how the mass $\theta_a+\theta_b$ splits between $a$ and $b$ into a single explicit function $H_s (\rho)=\E\big[\binom{s}{K}\big]$, where $K\sim\Bin(s,\rho)$. We then prove that $H_s (\rho)$ is a polynomial in $c:=\rho (1 - \rho)$ with nonnegative coefficients, hence is maximized exactly at $\rho = \tfrac12$ (for $s\ge2$, uniquely so). In summary, averaging two distinct coordinates of $\theta$ therefore strictly increases $\E_\theta[M(T)]$ whenever $d\ge2$. This forces any maximizer of $\E_\theta[M(T)]$ over the probability simplex to be $\theta_0$.
\end{itemize}

\section{Notation: partitions, patterns, and symmetric sums}\label{sec : notation}
Following~\cite{Group2026}, in this section, we introduce some notation and terminology that will be used throughout the article.

We first recall the partition set $\Pd$ and the counting vector $b(\ell)$ from Section~\ref{Sec : intro}. For $\ell\in\Pd$, set (with the convention $0!=1$, so trailing zero entries do not affect these)
$$
\binom{d}{\ell}:=\frac{d!}{\ell_1!\,\ell_2!\cdots\ell_d!},\qquad
\binom{q}{b(\ell)}:=\frac{q!}{b_0(\ell)!\,b_1(\ell)!\cdots b_d(\ell)!}.
$$
More generally, for $n=(n_\xi)_{\xi\in\Fq}\in\mathbb Z_{\ge0}^{\Fq}$ with $|n|:=\sum_\xi n_\xi=d$, we will use $\binom{d}{n}:=d!/\prod_\xi n_\xi!$.

\begin{definition}\label{def:pattern function}
We say that a nonnegative function $\alpha:\Fq \to \Z$ is a \emph{pattern-$\ell$ exponent} if exactly $b_j(\ell)$ of its values equal $j$ for each $0 \le j\le d$. For $x=(x_\xi)_{\xi\in\Fq}$, we write
$$
x^\alpha:=\prod_{\xi\in\Fq}x_\xi^{\alpha(\xi)} \qtq{and} m_\ell(x):=\sum_{\alpha\text{ pattern-}\ell}x^\alpha.
$$
Fix once and for all an enumeration $\Fq=\{\xi_1,\dots,\xi_q\}$, and let $S_q:=\Sym(\{1,\dots,q\})$ denote the symmetric group on $q$ letters. The \emph{$\ell$-th symmetric sum} is defined as
$$
\Sigma_\ell(x):=\sum_{\tau\in S_q}x_{\xi_{\tau(1)}}^{\ell_1}
x_{\xi_{\tau(2)}}^{\ell_2}\cdots x_{\xi_{\tau(d)}}^{\ell_d}.
$$
\end{definition}

Informally, one may think of a pattern-$\ell$ exponent function $\alpha$ as having the same distribution as $\ell$. Different pattern-$\ell$ exponents give different monomials $x^\alpha$, so $m_\ell(x)$ is a sum of distinct monomials. Note that there are $\binom{q}{b(\ell)}$ many pattern-$\ell$ exponents. A simple counting argument gives
\begin{equation}\label{E : sigma_l v m_l}
\Sigma_\ell(x) =  b(\ell)! \, m_\ell(x).
\end{equation}
Furthermore, for $t=(t_1,\dots,t_d)\in\Fq^d$, define its count vector by  $r(t)=(r(t)_\xi)_{\xi\in\Fq}\in \mathbb Z_{\ge0}^{\Fq}$, where 
\begin{equation}\label{E : def count vec}
r(t)_\xi:=\#\{i:t_i=\xi\}.
\end{equation}

\begin{remark}
The terminology above is only for convenience: a pattern-$\ell$ exponent is a rearrangement of the partition $\ell$ padded with $q-d$ zeros, so that $m_\ell$ is the \emph{monomial symmetric polynomial} attached to the partition $\ell$ in the variables $(x_\xi)_{\xi\in\Fq}$, in the standard notation of \cite[Ch.~I, \S2]{Macdonald1995}. In the notation $\ell=(1^{b_1(\ell)}2^{b_2(\ell)}\cdots)$ of \cite[Ch.~I, \S1]{Macdonald1995}, the entries of $b(\ell)$ are the multiplicities of the parts of $\ell$, and $\binom{q}{b(\ell)}$ is the size of the $S_q$-orbit of the corresponding exponent vector.
\end{remark}

We note the following two identities that will be used in the next section.
\begin{lemma}\label{L : m_l v eta}
For each $x = (x_\xi)_{\xi \in \Fq}$, the following holds.
\begin{enumerate}[label=(\roman*)]
\item $\displaystyle\sum_{\ell\in\Pd}\binom{d}{\ell}m_\ell(x)=(\sum_{\xi\in\Fq}x_\xi)^d$;
\item $\displaystyle\sum_{\ell\in\Pd}\binom{d}{\ell}^2m_\ell(x)
=\sum_{t\in\Fq^d}\binom{d}{r(t)}\prod_{i=1}^dx_{t_i}$.
\end{enumerate}
\end{lemma}

\begin{proof}
First, we prove (i). We have
$$
\Big(\sum_{\xi\in\Fq}x_\xi\Big)^d=\sum_{t=(t_1,\dots,t_d)\in\Fq^d} x^{r(t)} = \sum_{r \in\mathbb Z_{\ge0}^{\Fq} :\,|r|=d}\binom{d}{r}\,x^r.
$$
For each $r$ in the above sum, let $\ell(r)\in\Pd$ be the non-increasing rearrangement of the nonzero values of $r$, padded with zeros to a length of $d$. We have $\binom{d}{r}=\binom{d}{\ell(r)}$. Moreover, for a fixed $\ell$, the set of $r$ with $\ell(r)=\ell$ is exactly the set of pattern-$\ell$ exponents, so the above RHS is
$$
\sum_{\ell\in\Pd}\ \sum_{r :\,\ell(r)=\ell}\binom{d}{\ell} x^r
=\sum_{\ell\in\Pd}\binom{d}{\ell}\sum_{\alpha\text{ pattern-}\ell}x^\alpha
=\sum_{\ell\in\Pd}\binom{d}{\ell}\,m_\ell(x).
$$
This establishes (i).

We proceed to establish (ii). Arguing as in the previous part, we obtain
$$
\sum_{t\in\Fq^d}\binom{d}{r(t)}\prod_i x_{t_i}
=\sum_{t\in\Fq^d}\binom{d}{r(t)} x^{r(t)}
=\sum_{r \in\mathbb Z_{\ge0}^{\Fq} :\,|r|=d}\binom{d}{r}^2 x^r,
$$
and, by the same reasoning as above, the above RHS can be rewritten as
$$
\sum_{\ell\in\Pd}\binom{d}{\ell}^2\, m_\ell(x),
$$
finishing the proof.
\end{proof}

\section{Reduction to a positivity statement}

In this section, we show that Conjecture~\ref{Conj1} is equivalent to the non-negativity of a certain symmetric form, together with an explicit description of its zero set; see Proposition~\ref{Prop : equiv}. This corresponds to the content of \cite[\S2.1--\S2.2]{Group2026}. For the reader’s convenience, we reproduce the derivation here.

\subsection{The \texorpdfstring{$L^{2d}$}{L\^{}2d} norm as a sum over fibers}
First, we note that
\begin{equation}\label{E : ortho-expan}
\|(f\sigma)^\vee\|_{L^{2d}(\Fq^d,\d x)}^{2d}
=\frac1{q^d}\sum_{\zeta\in\Fq^d}\Big|\sum_{\substack{\xi_1+\dots+\xi_d=\zeta\\
\xi_i\in\Gamma}}\ \prod_{i=1}^d f(\xi_i)\Big|^2.
\end{equation}
We use an argument in the proof of \cite[Proposition 2.1]{RuquelmeSilva2024} that utilizes the orthogonality of additive characters of finite fields. Indeed, for each fixed $x$, we have
$$
q^{2d}\,|(f\sigma)^\vee(x)|^{2d}=\sum_{\xi_1,\dots,\xi_d\in\Gamma}\ \sum_{\eta_1,
\dots,\eta_d\in\Gamma}\Big(\prod_if(\xi_i)\Big)\Big(\prod_i\overline{f(\eta_i)}
\Big)\,e\Big(x\cdot\big(\textstyle\sum_i\xi_i-\sum_i\eta_i\big)\Big).
$$
Summing over $x$ and then applying the orthogonality of additive characters of finite fields (see~\cite[Chapter 5, \S1]{Lidlbook1997}), we get
$$
\sum_{x\in\Fq^d}|(f\sigma)^\vee(x)|^{2d}=\frac1{q^d}\!\!\sum_{\substack{
\xi_1,\dots,\xi_d,\eta_1,\dots,\eta_d\in\Gamma\\ \sum_i\xi_i=\sum_i\eta_i}}
\!\!\Big(\prod_if(\xi_i)\Big)\Big(\prod_i\overline{f(\eta_i)}\Big).
$$
For each fixed $\zeta\in\Fq^d$, the tuples $\xi$ with
$\sum_i\xi_i=\zeta$ and the tuples $\eta$ with $\sum_i\eta_i=\zeta$ range independently over the same set $\{t\in\Gamma^d:\sum_it_i=\zeta\}$. Thus, the above RHS is
$$
\frac1{q^d} \sum_{\zeta\in\Fq^d}\ \Big(\sum_{\sum_i\xi_i=\zeta}\prod_if(\xi_i)\Big)
\Big(\sum_{\sum_i\eta_i=\zeta}\prod_i\overline{f(\eta_i)}\Big)
= \frac1{q^d} \sum_{\zeta\in\Fq^d}\Big|\sum_{\substack{\xi_1+\dots+\xi_d=\zeta\\\xi_i\in
\Gamma}}\prod_if(\xi_i)\Big|^2
$$
establishing~\eqref{E : ortho-expan}.

\subsection{Newton's identities and the fiber structure}\label{subsec: newton identity}

This leads us to examine in more detail the index set appearing in the summand on RHS\eqref{E : ortho-expan}. Fixing $\zeta\in\Fq^d$, we consider the solutions $(t_1,\dots,t_d)\in \Fq^d$ of the system
\begin{equation}\label{E : system}
t_1^k+t_2^k+\dots+t_d^k=\zeta_k,\qquad k=1,\dots,d.
\end{equation}
For $t=(t_1,\dots,t_d)\in\Fq^d$, let $p_k(t):=\sum_it_i^k$ denote the $k$-th power sum and let $e_0(t):=1,e_1(t),\dots,e_d(t)$ be the elementary symmetric polynomials defined by expanding the product
\begin{equation}\label{E : elem-sym}
\prod_{i=1}^d(X-t_i)
=X^d-e_1(t)X^{d-1}+e_2(t)X^{d-2}-\dots+(-1)^de_d(t).
\end{equation}
We recall that the classical Newton identities (see \cite[page 81]{Waerden1953}) relate the power sums $p_k$ to the elementary symmetric polynomials via
\begin{equation}\label{eq:newton}
p_k(t)=\sum_{i=1}^{k-1}(-1)^{i-1}e_i(t)\,p_{k-i}(t)+(-1)^{k-1}k\,e_k(t),
\qquad k=1,\dots,d
\end{equation}
(for $k = 1$, the sum on the right is empty; hence $0$, giving $p_1(t)=e_1(t)$).

Our strategy is to associate each $\zeta \in \Fq^d$ with a polynomial analogous to RHS\eqref{E : elem-sym}, and then compare it with the polynomial in \eqref{E : elem-sym} associated with a solution $t \in \Fq^d$ of \eqref{E : system}. To this end, we now imitate the above identity to construct a bijection $\Psi:\Fq^d\to\Fq^d$. For $\zeta \in \Fq^d$, we define $\Psi(\zeta) := (\eta_1,\dots,\eta_d)$ recursively as follows. First, set $\eta_1 := \zeta_1$, and then, for $1<k\le d$, assuming that $\eta_1,\dots,\eta_{k-1}$ have already been uniquely determined, we define $\eta_k$ via
\begin{equation}\label{E : def eta_k}
\zeta_k =: \sum_{i=1}^{k-1}(-1)^{i-1} \eta_i\,\zeta_{k-i}+(-1)^{k-1}k\, \eta_k.
\end{equation}
Finally, we consider the map $\zeta\mapsto P_\zeta$, where 
$$
P_\zeta(X):=X^d - \eta_1X^{d-1}+\eta_2X^{d-2}-\dots+(-1)^d \eta_d.
$$

We observe that $t=(t_1,\dots,t_d)\in\Fq^d$ solves the system~\eqref{E : system} precisely when $P_\zeta(X)=\prod_{i=1}^d(X-t_i)$. This is a direct consequence of the bijectivity of $\Psi$, which holds since $p > d$  and thus each $1 \leq k \leq d$ is invertible in $\Fq$. Indeed, $t$ satisfies~\eqref{E : system} if and only if $p_i(t)=\zeta_i$, which is via $\Psi$ equivalent to $e_i(t)=\eta_i$, thereby proving the assertion.

Looking back at RHS\eqref{E : ortho-expan} and in view of the above observation, we now look at the set of all $\zeta \in \Fq^d$ for which the polynomial $P_\zeta$ splits completely over $\Fq$. Write $P_\zeta(X)=\prod_{r=1}^{k}(X-c_r)^{\ell_r}$ with the  $c_r$ distinct and $\ell_1\ge\dots\ge\ell_k\ge1$. Note that $\sum_r\ell_r=d$. We pad $(\ell_1,\dots,\ell_k)$ with $d-k$ trailing zeros to obtain a $d$-tuple $\ell=(\ell_1,\dots,\ell_k,0,\dots,0)\in\Pd$.

\begin{definition}\label{def:Wl}
For $\ell\in\Pd$, let $W_\ell\subset\Fq^d$ be the set of all $\zeta \in \Fq^d$ such that $P_\zeta$ splits completely over $\Fq$ with the multiplicity pattern $\ell$. 
\end{definition}

For $\ell\in\Pd$ and $\zeta\in W_\ell$, suppose that $P_\zeta=\prod_{r=1}^k (X-c_r)^{\ell_r}$. Then the solution set of system \eqref{E : system} is precisely the set of all orderings of the multiset that contains $\ell_r$ copies of $c_r$ for each $r=1,\dots,k$. Note that there are $\binom{d}{\ell}$ such orderings. For any $\zeta\in W_\ell$ and $f:\Gamma\to\mathbb C$, every $t$ solving \eqref{E : system} produces, via $\xi_i:=\gamma(t_i)$, the same value of the product $\prod_{i=1}^d f(\xi_i)$. We therefore set
$$
\pi_\zeta(f):=\prod_{i=1}^d f(\gamma(t_i))\qquad\text{for any solution } t \,\, \text{of} \,\, \eqref{E : system}.
$$

Next, for each fixed $\zeta\in W_\ell$, consider the associated nonnegative function $\Phi_\zeta: \Fq\to \Z$ by letting $\Phi_\zeta(\xi)$ be the multiplicity of $\xi$ as a root of $P_\zeta$. Observe that the assignment $\zeta \mapsto \Phi_\zeta$ is a bijection from $W_\ell$ onto the set of pattern-$\ell$ exponents; in particular, we have $|W_\ell|=\binom{q}{b(\ell)}$. Moreover, for every function $f:\Gamma\to\mathbb C$ and each point $\zeta\in W_\ell$, we obtain
\[
|\pi_\zeta(f)|^2 = x(f)^{\Phi_\zeta},
\]
where $x(f):=\big(|f(\gamma(\xi))|^2\big)_{\xi\in\Fq}\in[0,\infty)^{\Fq}$. Consequently, it follows that
$$
\sum_{\zeta\in W_\ell}|\pi_\zeta(f)|^2 = m_\ell(x(f)).
$$
In conclusion, combining this with  RHS\eqref{E : ortho-expan} we deduce that
\begin{equation}\label{E : LHS-final}
\|(f\sigma)^\vee\|_{2d}^{2d}=\frac1{q^d}\sum_{\ell\in\Pd}\binom{d}{\ell}^2 m_\ell(x(f)).
\end{equation}

\begin{remark}\label{rmk : modulus}
The above shows that $\|(f\sigma)^\vee\|_{L^{2d}(\Fq^d,\d x)}$ depends on $f$ only through $|f|$. The reason is that no cancellation can occur in the fiber sums of~\eqref{E : ortho-expan}: For $\zeta\in W_\ell$, all $\binom{d}{\ell}$ solutions of~\eqref{E : system} are reorderings of each other, so the inner sum is $\binom{d}{\ell}$ times the single complex number $\pi_\zeta(f)$. This is special to the exponent $2d$, at which the fibers of $(\xi_1, \dots, \xi_d) \mapsto \sum_i \xi_i$ over $\Gamma$ consist of a single $S_d$-orbit. It also explains why the maximizers in Theorem~\ref{T : main thm} form such a large set, in contrast with the euclidean setting, where they usually constitute a finite-dimensional family generated by the symmetries of the problem.
\end{remark}

\subsection{The equivalence}

For each $\ell\in\Pd$, set
$$
\omega_\ell:=\binom{d}{\ell}\binom{q}{b(\ell)}\Big(A-\binom{d}{\ell}\Big),
$$
and for $x\in[0,\infty)^{\Fq}$, define
$$
Q(x):=\sum_{\ell\in\Pd}\omega_\ell\,\Sigma_\ell(x),
$$
where $A$ is as introduced in \eqref{eq_A}. Then we obtain the following equivalence.

\begin{proposition}\label{Prop : equiv}
Let $d\ge2$ and $p>d$. Then, inequality \eqref{E : conj1} holds for every $f:\Gamma\to\mathbb C$ if and
only if $Q(x)\ge0$ for every $x\in[0,\infty)^{\Fq}$. Furthermore, if this holds, a nonzero $f$ attains equality in \eqref{E : conj1} if and only if $Q(x(f))=0$.
\end{proposition}

\begin{proof}
Let $f:\Gamma\to\mathbb C$ and define $x:=x(f)\in[0,\infty)^{\Fq}$. Looking back at RHS\eqref{E : conj1}, Lemma~\ref{L : m_l v eta}(i) yields
$$
\|f\|_{L^2(\sigma)}^{2d} = \frac1{q^d} \sum_{\ell\in\Pd} \binom{d}{\ell} m_\ell(x).
$$
Hence, combining this with identity~\eqref{E : LHS-final} and identity~\eqref{E : sigma_l v m_l}, we see that \eqref{E : conj1} is valid exactly when $Q(x)\ge0$. The statement in the equality case follows immediately.
\end{proof}

\section{Probabilistic reformulation}\label{sec : prob reformulation}

The reduction in Proposition~\ref{Prop : equiv} leaves us with the positivity of the symmetric form $Q$. The purpose of this section is now to provide a slightly more conceptual reformulation of this condition. Namely, we show that $Q(x)\ge0$ is equivalent to a purely probabilistic statement: among all probability measures on $\F_q$, the uniform measure maximizes the expected number of distinct rearrangements of an i.i.d. sample (see Corollary~\ref{cor:equivalence}).

We first recall the count vector $r(t)$ from~\eqref{E : def count vec}. For $t\in\Fq^d$, set
\begin{equation}\label{E : def M(t)}
M(t):=\binom{d}{r(t)}.
\end{equation}
Let $\Simp:=\{\theta\in[0,\infty)^{\Fq}:\sum_\xi\theta_\xi=1\}$ be the set of all probability measures on $\Fq$. For a random variable $Z$ taking values in $\Fq$, we say it has law $\theta$ on $\Fq$ if $\Prob(Z = \xi) = \theta_\xi$ for each $\xi\in\Fq$. Let $T=(T_1,\dots,T_d)$ be a tuple consisting of i.i.d. random variables, with law $\theta$ on $\Fq$. Then
\begin{equation}\label{eq:EM-def}
\E_\theta[M(T)]=\sum_{t\in\Fq^d}M(t)\,\Prob(T=t)=\sum_{t\in\Fq^d}
\binom{d}{r(t)}\prod_{i=1}^d\theta_{t_i}.
\end{equation}

\begin{proposition}\label{Prop : prob form}
For any nonzero $x\in[0,\infty)^{\Fq}$, set $s:=\sum_\xi x_\xi$ and $\theta:=x/s\in\Simp$. Then
$$
Q(x) = q! \, s^d \big(A - \E_{\theta} \big[M(T)\big]\big).
$$
In addition, $A=\E_{\theta_0}\big[M(T)\big]$ where $\theta_0:=(1/q,\dots,1/q)$ denotes the uniform measure.
\end{proposition}

\begin{proof}
Applying identity~\eqref{E : sigma_l v m_l} first, then Lemma~\ref{L : m_l v eta}, and finally substituting $x=s\theta$, we obtain
$$
Q(x)=q!\sum_{\ell\in\Pd}\binom{d}{\ell}\Big(A-\binom{d}{\ell}\Big)m_\ell(x)
= q!\Big[As^d-\sum_{t\in\Fq^d}\binom{d}{r(t)}\prod_{i=1}^d x_{t_i}\Big]
= q!\,s^d\big(A-\E_\theta[M(T)]\big).
$$
We now evaluate $\E_{\theta_0}\big[M(T)\big]$. By Lemma~\ref{L : m_l v eta}(ii), this is $\sum_{\ell\in\Pd}\binom{d}{\ell}^2m_\ell(\theta_0)$. Now, for each pattern-$\ell$ exponent $\alpha$, we have $\theta_0^\alpha = q^{-d}$, and so $m_\ell(\theta_0) = \binom{q}{b(\ell)} q^{-d}$. As a consequence, 
$$
\E_{\theta_0}[M(T)] = \frac1{q^d} \sum_{\ell\in\Pd} \binom{d}{\ell}^2\binom{q}{b(\ell)} = A.
$$
\end{proof}

\begin{remark}\label{rmk : size of A}
Proposition~\ref{Prop : prob form} identifies the optimal constant as an expected value, which makes its size transparent. Since $M(t)\le d!$ for every $t \in \Fq^d$, with equality precisely when the entries of $t$ are distinct, and since the tuples with distinct entries have probability $\prod_{i=1}^{d-1}(1-i/q)$ under $\theta_0$, we get
$$
d! \prod_{i = 1}^{d - 1} \Big(1 - \frac{i}{q}\Big) \; \le \;A\; < \;d!.
$$
In particular the optimal constant $A \to d!$ as $q \to \infty$ for each fixed $d$, consistently with $A = 2 - 1/q$ for $d = 2$; see Example~\ref{ex:small-d} below.
\end{remark}

Since $Q(0)=0$, combining the above with the first equivalent formulation, Proposition~\ref{Prop : equiv}, yields the following equivalence.
\begin{corollary}\label{cor:equivalence}
Conjecture~\ref{Conj1} can be restated equivalently as follows: For every $\theta\in\Simp$, we have $\E_\theta[M(T)]\le A$, and equality holds if and only if $\theta=\theta_0$.
\end{corollary}

\begin{example}\label{ex:small-d}

We explicitly evaluate the terms in Corollary~\ref{cor:equivalence} and prove the conjecture directly for $d=2$. Let $t=(t_1,t_2)\in\Fq^2$. If $t_1\neq t_2$, then $r(t)$ has two coordinates equal to $1$ and all others $0$, hence $M(t)=2!/(1!\,1!)=2$. If instead $t_1=t_2$, then $r(t)$ has a single coordinate equal to $2$ and the remaining coordinates $0$, so $M(t)=2!/2!=1$. Either way, $M(t)=2-\rchi_{\{t_1=t_2\}}$. Consequently,
$$
\E_\theta[M(T)]
=\E_\theta \big(2-\rchi_{\{T_1=T_2\}}\big)
=2-\Prob(T_1=T_2)
=2-\sum_{\xi\in\Fq}\theta_\xi^2.
$$
Since $\theta$ is a probability measure, by Cauchy--Schwarz, $\sum_\xi \theta_\xi^2\ge 1/q$, with equality exactly when $(\theta_\xi)_\xi$ is proportional to $(1,\ldots,1)$, i.e., when $\theta=\theta_0$. In other words, $\E_\theta[M(T)]\le 2-1/q$, with equality if and only if $\theta=\theta_0$. This agrees with $A=2-1/q$, which can also be obtained directly. Indeed, we have $\P_2=\{(1,1),(2,0)\}$, $\binom{q}{b((1,1))}=\binom{q}{2}$ and $\binom{q}{b((2,0))}=q$, so
$$
A=q^{-2}\Big[2^2\cdot \tfrac{q(q-1)}2+ q\Big]=2- 1/q.
$$
\end{example}

\begin{remark} The inequality in Corollary~\ref{cor:equivalence} is a special case of \cite[Theorem 3.6]{Peskir1995}, where the stronger statement that $\theta \mapsto \E_\theta[M(T)]$ is Schur-concave on $\Delta_q$ is established via a transfer theorem for the multinomial distribution \cite{Rinott1973}. However, the characterization of the equality case in Conjecture \ref{Conj1} does not seem to appear there. In the next section, we present a self-contained proof of this inequality that automatically yields the characterization of the equality case. We refer to \cite[Ch.~3]{Marshal2011} for background on majorization and Schur-convexity, and note that the analog of Schur-concavity for Rademacher sums goes back to \cite{Eaton1970, Efron1969}.
\end{remark}

\section{An extremal property of the uniform distribution}\label{sec : proof of monotonicity}

Throughout this section, $\mathcal X$ denotes a finite set with $q$ elements with $q > 1$, and $\Simp$ denotes the set of probability distributions on $\mathcal X$. No arithmetic structure is imposed on $\X$, and the reader interested only in Theorem~\ref{T : main thm} may take $\mathcal X = \Fq$ throughout. The count vector $r(t)$ of~\eqref{E : def count vec} and the rearrangement count $M(t)=\binom{d}{r(t)}$ of~\eqref{E : def M(t)} are defined verbatim for $t\in\X^d$, and we use them in that generality without further comment. Our goal is to show that $\E_\theta[M(T)]$, the expected number of distinct rearrangements of the sample, strictly increases under a natural symmetrization of $\theta$.

\begin{proposition}\label{prop:pinch}
Let $d\ge2$, let $\theta\in\Simp$, and let $a,b\in\X$ satisfy $\theta_a\ne\theta_b$. Let $\theta'\in\Simp$ agree with $\theta$ off $\{a,b\}$, that is, $\theta'_\xi=\theta_\xi$ for $\xi\ne a,b$, and set $\theta'_a=\theta'_b:=\tfrac12(\theta_a+\theta_b)$. Then
$$
\E_{\theta'}[M(T)]>\E_\theta[M(T)].
$$
\end{proposition}

The passage from $\theta$ to $\theta'$ is, in the language of majorization \cite[Chapter 2]{Marshal2011}, a $T$-transform with parameter $\lambda=1/2$. In particular, $\theta' \prec \theta$. Given the above monotonicity result, the proof of Conjecture~\ref{Conj1} is immediate. Indeed, by Corollary~\ref{cor:equivalence}, it suffices to prove the following for $\X = \Fq$.

\begin{proof}[\textbf{Proof of Theorem~\ref{T : prob thm}}]
Being a polynomial, the function $\theta\mapsto\E_\theta[M(T)]$ is continuous, so it has a maximum point on the compact set $\Simp \subset \R^{|\X|}$. The fact that the uniform distribution $\theta_0$ is the unique maximizer follows directly from Proposition~\ref{prop:pinch}.
\end{proof}

We now return to the proof of Proposition~\ref{prop:pinch}. For the remainder of the section, we fix the two elements $a, b\in\X$. The strategy is to isolate the dependence of $\E_\theta[M(T)]$ on the way the mass $\theta_a + \theta_b$ is divided between $a$ and $b$. Identifying $a$ and $b$ as a single letter factors the rearrangement count exactly, as the rearrangement count of the merged sample times a binomial coefficient that captures all of that dependence. Taking expectations converts the binomial coefficient into a single one-variable function $H_s$, and the proposition reduces to the statement that $H_s$ is largest when the mass is divided evenly.

\subsection{An exact factorization under merging two letters}

Let $*$ be a new symbol not in $\X$. Consider the set $\A:=(\X \setminus\{a,b\})\cup\{*\}$ of size $q - 1$, and let $\pi:\X\to\A$ be the ``quotient" map that identifies $a$ with $b$, i.e., $\pi(a)=\pi(b)=*$ and $\pi(\xi)=\xi$ for $\xi \neq a, b$. For $T=(T_1,\dots,T_d)$ i.i.d.\@ with law $\theta$, set $\widetilde T_i:=\pi(T_i)$. Write $N_\xi:=\#\{i:T_i=\xi\}$ for $\xi\in\X$ (so $N=r(T)$ in the notation of Lemma~\ref{L : m_l v eta}), and let $S:=N_a+N_b$.

For every $t\in\X^d$, with $\tilde t:=(\pi(t_1),\dots,\pi(t_d))$ and $s:=r(t)_a+r(t)_b$, it holds
\begin{equation}\label{E : M(t) vs M(tilde t)}
M(t)=\binom{s}{r(t)_a}\,M(\tilde t).
\end{equation}
Indeed, note that the count vector $r(\tilde t)$ (over the set $\A$) satisfies $r(\tilde t)_\xi=r(t)_\xi$ for every $\xi\ne a, b$ and $r(\tilde t)_* = r(t)_a + r(t)_b = s$. Hence,
$$
M(\tilde t) = \frac{d!}{s!\prod_{\xi\ne a,b}r(t)_\xi!} = \binom{d}{r(t)} \frac{r(t)_a!\,r(t)_b!}{s!}.
$$
Since $M(t) = \binom{d}{r(t)}$, identity~\eqref{E : M(t) vs M(tilde t)} follows. Applying this pointwise to the tuple $T$ gives the following almost sure identity
\begin{equation}\label{E : factor-random}
M(T)=\binom{S}{N_a}\,M(\widetilde T).
\end{equation}

\begin{lemma}\label{lem:conditional-binom}
Suppose $m:=\theta_a+\theta_b>0$ and let $ \rho :=\theta_a/m$. Then, conditionally on $\widetilde T$, the random variable $N_a$ has a $\Bin(S, \rho)$ distribution.
\end{lemma}

This is the standard \emph{lumping} property of the multinomial distribution: merging two cells again produces a multinomial vector, and conditionally on the merged cell count, the split between the two original cells is binomial (see, e.g., \cite[Chapter 35]{JKB1997}). We include the following short proof to keep the account self-contained.

\begin{proof}
To begin, we observe that $T_1,\ldots,T_d$ are conditionally independent given $\widetilde T$. Indeed, for a fixed $t\in\X^d$ and $u \in\A^d$ such that $\pi(t_i) = u_i$ for every $i$ (otherwise, both sides of the identity below are zero), by the independence of $T_i$, we have
$$
\Prob(T=t\mid\widetilde T = u)
=\prod_{i=1}^d\frac{\Prob(T_i=t_i)}{\Prob(\widetilde T_i=u_i)}
=\prod_{i=1}^d\Prob(T_i=t_i\mid\widetilde T_i=u_i),
$$
because each term on the RHS equals $\theta_{t_i}/m$ when $t_i\neq u_i$ (that is, $t_i\in\{a,b\}$ and $u_i=*$), and equals $1$ in all other cases. In consequence, given $\tilde T$, the $S$ coordinates with $u_i=*$ are i.i.d.\@ with $\Prob(T_i=a\mid\widetilde T=u)=\rho$, giving us $N_a \sim \Bin(S,\rho)$.
\end{proof}

Define the probability measure $\tilde \theta$ on $\A$ by merging the probabilities at $a$ and $b$; i.e., set $\tilde\theta(*) := m$ and $\theta_\xi$ otherwise. Thus, the projected random tuple $\widetilde T$ is an i.i.d.\@ tuple with law $\tilde\theta$. Now, for a nonnegative integer $s$ and $\rho \in[0,1]$, set
\begin{equation}\label{eq:Hs-def}
H_s (\rho):=\E\Big[\binom{s}{K}\Big],\qquad K\sim\Bin(s, \rho).
\end{equation}
In other words, $H_s(\rho) = \sum_{k=0}^s \binom sk^2 \rho^k (1- \rho)^{s-k}$. Here we use
$$
\binom{0}{0} := 1 \qtq{and} \binom{0}{k} := 0 \qtq{for each positive integer $k$,}
$$
and we will follow this for the remainder of this article.

\begin{proposition}\label{prop:factor-expectation}
Let $\theta\in\Simp$ with $m=\theta_a+\theta_b>0$ and $ \rho = \theta_a/m$. Then
$$
\E_\theta[M(T)]=\E_{\tilde\theta}\big[M(\widetilde T)\,H_S(\rho)\big],
$$
where $S=S(\widetilde T)$ is the number of coordinates of $\widetilde T$ equal to $*$.
\end{proposition}

\begin{proof}
Taking expectations on both sides of the identity \eqref{E : factor-random}, we get
$$
\E_\theta[M(T)]=\E\Big[\binom{S}{N_a}\,M(\widetilde T)\Big].
$$
We recall that for conditional expectation $\E[X]=\E\big[\E[X\mid \mathcal F]\big]$ for any random variable $X$ and any $\sigma$-field $\mathcal F$. Setting $X=\binom{S}{N_a}M(\widetilde T)$ and $\mathcal F$ generated by $\widetilde T$, we get
$$
\E\Big[\binom{S}{N_a}M(\widetilde T)\Big]
=\E\Big[\E\big[\tbinom{S}{N_a}M(\widetilde T)\mid\widetilde T\big]\Big]
=\E\Big[M(\widetilde T)\,\E\big[\tbinom{S}{N_a}\mid\widetilde T\big]\Big].
$$
At this point, we now apply Lemma~\ref{lem:conditional-binom}. Thus, conditionally on $\widetilde T$ (so that $S$ is fixed), $N_a\sim\Bin(S, \rho)$, and so $\E\big[\tbinom{S}{N_a}\mid\widetilde T\big]=H_S(\rho)$. Consequently, we have 
$$
\E_\theta[M(T)]=\E\big[M(\widetilde T)\,H_S (\rho)\big].
$$
We recall that the projected i.i.d.\@ tuple $\widetilde T$ has law $\tilde\theta$, so the above RHS is $\E_{\tilde\theta}[M(\widetilde T)H_S(\rho)]$.
\end{proof}

\subsection{Monotonicity} Proposition~\ref{prop:factor-expectation} reduces the effect of averaging the masses at $a$ and $b$ to the one-variable function $H_s$. We are thus led to examine the function $H_s$. We claim that for a fixed $s \ge 0$ as a function of $c := \rho (1 - \rho) \in[0, 1/4]$, $H_s$ is non-decreasing, furthermore, for $s\ge2$, it is strictly increasing. So the maximum is attained at $\rho = 1/2$. This would follow once we establish that $H_s$, as a polynomial in $c$, has nonnegative coefficients with at least one nonzero coefficient for $s \geq 2$. In fact, we will show that for $\rho \in [0,1]$
\begin{equation}\label{E : pos coeff}
H_s(\rho)
= \sum_{j=0}^{\lfloor s/2\rfloor} \binom{s}{2j}\binom{2j}{j}\,c^j.
\end{equation}

To see this, set $u :=\sqrt{\rho}$ and $v :=\sqrt{1-\rho}$. Then $\big|u + v e^{i\varphi}\big|^{2} = 1+2\sqrt{c}\,\cos\varphi$, and so
$$
\big(1+2\sqrt{c}\cos\varphi\big)^s 
= \big(u + v e^{i\varphi}\big)^s \, \overline{\big(u + v e^{i\varphi}\big)^{s}}.
$$
Expanding each factor by the binomial theorem and using the orthogonality $\frac1{2\pi} \int_{0}^{2\pi} e^{i(k - l)\varphi} \, \d\varphi=\delta_{kl}$, we deduce
\begin{equation}\label{eq:intrep}
\frac{1}{2\pi}\int_{0}^{2\pi}\big(1+2\sqrt{c}\,\cos\varphi\big)^{s}\,\d\varphi
=\sum_{k=0}^{s}\binom{s}{k}^{2}u^{2(s-k)} v^{2k}
= H_s(\rho).
\end{equation}
We now expand the integrand in the LHS using the binomial theorem and noting that $\frac{1}{2\pi}\int_{0}^{2\pi}\cos^{m}\varphi\,\d\varphi=0$ for $m$ odd, and $=\binom{2j}{j}4^{-j}$ for $m = 2 j$, we deduce that the above LHS is 
$$
\sum_{j=0}^{\lfloor s/2\rfloor}\binom{s}{2j}\big(2\sqrt{c}\big)^{2j}\binom{2j}{j}4^{-j}
=\sum_{j=0}^{\lfloor s/2\rfloor}\binom{s}{2j}\binom{2j}{j}\,c^{\,j}
$$
giving the desired identity.

\begin{remark}
Identity~\eqref{E : pos coeff} relates $H_s$ to the Legendre polynomial of degree $s$. For $\rho\neq 1/2$, put $z:=(1-2\rho)^{-1}$, so that $z^{2}-1=4cz^{2}$; comparison of~\eqref{E : pos coeff} with the classical expansion $P_s(z)=\sum_{j\le s/2}\binom{s}{2j}\binom{2j}{j}\big(\tfrac{z^{2}-1}{4}\big)^{j}z^{\,s-2j}$ yields $P_s(z)=z^{s}H_s(\rho)$, and~\eqref{eq:intrep} is then Laplace's integral representation of $P_s$; see \cite[(4.8.10)]{Szego}.
\end{remark}

\subsection{The pinch: Proof of Proposition~\ref{prop:pinch}}

We can now complete the pinch argument. Proposition~\ref{prop:factor-expectation} expresses expectation in terms of the projected distribution. In addition, the monotonicity proved above shows that balancing the two masses at $a$ and $b$ can only increase the value.

\begin{proof}
We apply Proposition~\ref{prop:factor-expectation} to both $\theta$ and $\theta'$ and then use the monotonicity of $H_s$ to obtain
$$
\E_{\theta'}[M(T)]-\E_\theta[M(T)]
=\E_{\tilde\theta}\Big[M(\widetilde T)\big(H_S(1/2) - H_S(\rho)\big)\Big] \geq 0,
$$
where we recall that $\rho = \frac{\theta_a}m \neq \frac 12$.

To show that the above RHS is positive, it suffices to produce an event of positive probability on which the summand is positive. Since the projected $\widetilde T_i$ are i.i.d.\@ with $\Prob(\widetilde T_i=*)=m>0$, the event that every one of the $d$ coordinates equals $*$ has probability $\Prob(S = d) = m^d$. On this event $S= d \ge 2$, the strict monotonicity of $H_s$ in $c = \rho (1 - \rho)$ implies $H_d(1/2) > H_d(\rho)$, completing the proof.
\end{proof}

\begin{remark}
In~\cite{Group2026}, our earlier partial progress on Conjecture~\ref{Conj1} established $Q\ge0$ by a different route, namely, writing $Q=\sum_{n\in N}
\omega_n \Sigma_n+\sum_{p\in P}\omega_p\Sigma_p$ (splitting $\Pd$ into the partitions with negative and nonnegative weight $\omega_\ell$) and applying Muirhead's inequality term by term via a fractional matching between $N$ and $P$ (a weighted bipartite graph whose existence is characterized by Strassen's theorem, see \cite[Lemma 5]{Group2026}). This reduces Conjecture \ref{Conj1} to a finite combinatorial condition, which the authors verified by direct computation for $d\le20$ and showed that it holds once $q$ is large enough. The argument of the present note establishes $Q\ge0$ directly, through the probabilistic reformulation of Sections~\ref{sec : prob reformulation} and \ref{sec : proof of monotonicity}, without constructing any such matching. It says nothing about whether condition \cite[(2.19)]{Group2026} holds outside the ranges $d\le20$ or large $q$ already covered by \cite[Theorems 2--3]{Group2026}; that finite combinatorial question is, on its face, a different (and \emph{a priori}  stronger, since a Muirhead-type certificate is a sufficient but not necessary route to a symmetric form's positivity) statement from $Q\ge0$ itself and is not addressed by Theorem~\ref{T : main thm} or its proof.
\end{remark}

\section*{Acknowledgements}
This project originated during a SQuaRE at the American Institute of Mathematics. The authors are grateful to AIM for offering a supportive environment. Madrid was partially supported by the Simons Foundation Grant $\# 453576$. Oliveira e Silva was funded by FCT/Portugal and the Recovery and Resilience Plan (PRR) through projects UID/04459/2025 and UID/PRR/04459/2025, and by the project 2023.17881.ICDT (SHADE).
Stovall received partial support from NSF DMS-2246906 and from a grant from the Simons Foundation, SFI-MPS-SFM-00011865, and Tautges was partially supported by NSF DMS-2037851 and DMS-2246906. 
Part of this work was conducted during Biswas's visit to IIT Jammu, whose support and enriching environment he gratefully acknowledges.

\subsection*{AI Statement}
At various stages of this project, the authors made use of Claude. The starting point was a question one of us put to it: whether the central limit theorem approach in Beckner's proof of the sharp Hausdorff–Young inequality \cite[Proof of Theorem 1]{Beckner1975} could be adapted to the setting of Conjecture~\ref{Conj1}. The first proposal that came back rested on Schur concavity. When asked instead for a proof of the required monotonicity, along the lines of Beckner's argument and avoiding concavity arguments, the model produced a draft of what is now Section~\ref{sec : proof of monotonicity}, together with a rederivation of the reduction of \cite[\S2.1--\S2.2]{Group2026}.

We worked from that draft. The paper, as it stands, is our own, completely rewritten after checking the argument line by line and locating it in the literature.
The sources we found along the way are cited where they bear on the text. Responsibility for what appears here rests with the authors.

%%%%%%%%%%%%%%%%%%%%%%%%%%%%%%%%%%%%%%%%%%%%%%% 
%%%%%%%%%%%%%%%%%%%%%%%%%%%%%%%%%%%%%%%%%%%%%%% 
%%%%%%%%%%%%%%%%%%%%%%%%%%%%%%%%%%%%%%%%%%%%%%% 

\bibliographystyle{amsplain}
%bibliographystyle{abbrvnat}
\bibliography{const}

%%%%%%%%%%%%%%%%%%%%%%%%%%%%%%%%%%%%%%%%%%%%%%%
%%%%%%%%%%%%%%%%%%%%%%%%%%%%%%%%%%%%%%%%%%%%%%%
%%%%%%%%%%%%%%%%%%%%%%%%%%%%%%%%%%%%%%%%%%%%%%%

\end{document}